\documentclass[11pt]{article}

\usepackage{amsmath,amssymb,amsthm}
\usepackage[margin=1in]{geometry}
\usepackage[colorlinks=true,linkcolor=blue,citecolor=blue,urlcolor=blue]{hyperref}

\newtheorem{theorem}{Theorem}[section]
\newtheorem{proposition}[theorem]{Proposition}
\newtheorem{lemma}[theorem]{Lemma}
\newtheorem{corollary}[theorem]{Corollary}
\newtheorem{example}[theorem]{Example}
\theoremstyle{remark}
\newtheorem{remark}[theorem]{Remark}

\newcommand{\PP}{\mathbb P}
\newcommand{\EE}{\mathbb E}

\newcommand{\e}{\mathrm e}
\newcommand{\1}{\mathbf 1}
\newcommand{\cR}{\mathcal R}
\newcommand{\cS}{\mathcal S}

\title{On Arithmetic-Progression Biases in Random Partitions:\\
       Limiting Laws and the Li--Zhao Conjecture}
\author{Jiyou Li and Sicheng Zhao\\[2pt]
        School of Mathematical Sciences, Shanghai Jiao Tong University,\\
        Shanghai, 200240, China\\[2pt]
        }
\date{}

\begin{document}

\maketitle

\begin{abstract}
Let $\cR=\{r+jm:j\geq0\}$ and $\cS=\{s+jm:j\geq0\}$ be disjoint
arithmetic progressions of positive integers, and choose a partition of $n$
uniformly at random.  We determine the limiting probability that the number
of parts in $\cR$ is larger than the number of parts in $\cS$.  More
precisely, if $a=r/m$ and $b=s/m$, then this probability converges to
$I_{1/2}(a,b)$, the regularized incomplete beta function.  We prove the
stronger distributional statement that, after multiplication by the
partition saddle-point parameter, the difference of the two numbers of
parts converges to $m^{-1}(\log G_b-\log G_a)$, where $G_a$ and $G_b$ are
independent gamma variables.  The proof uses the Boltzmann conditioning
device for random partitions, a $q$-gamma evaluation of the unconditioned
characteristic function, and a uniform Fourier estimate showing that
conditioning on the total size does not change the limit.  We also evaluate
the finite-truncation limit studied by Li and Zhao.  The two evaluations
coincide, proving their conjecture on the commutation of the truncation and
large-size limits.
\end{abstract}



\section{Introduction}

Partition theory is one of the important subjects in combinatorics.  For a
positive integer $n$, a partition of $n$ is a non-increasing sequence of
positive integers whose sum is $n$.  We write a partition in frequency
notation as
\[
 \lambda=(1^{m_1}2^{m_2}\cdots),\qquad
 |\lambda|=\sum_{k\geq1}km_k,
\]
where $m_k$ is the multiplicity of the part $k$.  If $A$ is a set of
positive integers, then $N_A(\lambda)=\sum_{k\in A}m_k$ counts the parts
whose sizes lie in $A$, with repetitions included.

Recently biases in partitions have become an active subject in the study
of partitions.  By a bias, we mean the tendency of some parts to appear more
frequently than other parts.  Parity biases and biases between residue
classes have been studied, for example, in
\cite{Chern2022,KimKim2021,KimKimLovejoy2020}.  Related results for
unrestricted partitions, distinct partitions, and overpartitions were obtained
by Schlosser and Zhou \cite{SchlosserZhou2025}.  Andrews
\cite{Andrews1976} provides general background on partition theory, and
Li and Zhao \cite{LiZhao2026} provide further discussion of biases.

In their previous study, Li and Zhao considered three disjoint sets of
allowed parts $R,S,I$ and the proportion of admissible partitions for which
$N_R>N_S$.  In the case that $R,S,I$ are finite, they obtained an explicit
limiting ratio by comparing the number of integer points in a certain
polytope with its volume.

It is natural to consider the case that $R$ and $S$ are infinite arithmetic
progressions.  Let $C_{n,N}$ denote the corresponding proportion after the
two progressions are truncated to their first $N$ terms.  There are two
iterated limits:
\[
 \begin{array}{c}
 \text{remove the truncation }(N\to\infty)
 \text{ and then let the partition size grow }(n\to\infty),\\[2mm]
 \text{or first use the finite-dimensional asymptotic }(n\to\infty)
 \text{ and then let }N\to\infty.
 \end{array}
\]
The first one gives the bias for unrestricted partitions, while the second
one can be computed from the finite-dimensional formula.  Note that the
stabilization in $N$ depends on $n$, so the equality of these two limits does
not follow directly.  Li and Zhao conjectured that they are equal
\cite[Conjecture 1]{LiZhao2026}.

Li and Zhao state their conjecture with the normalization $(r,s,m)=1$.
We retain this normalization because the finite-truncation input used below
is their finite-set theorem, whose marked parts must have greatest common
divisor one.  Indeed, for every $N\geq2$,
\[
 \gcd(\cR_N\cup\cS_N)=\gcd(r,s,m)=1,
\]
so the cited theorem applies to all sufficiently large $N$; the finitely
many smaller values of $N$ do not affect the limit $N\to\infty$.  The
unrestricted Boltzmann argument itself only depends on $r,s,m$ through
$a=r/m$ and $b=s/m$, but the commutation statement is made under the
normalization above.

We prove the conjecture and identify the limiting distribution of the
difference of the two numbers of parts.  Fix positive integers $r,s,m$ with
\begin{equation}\label{eq:distinct-residues}
 r\not\equiv s\pmod m,
\end{equation}
and
\begin{equation}\label{eq:coprime-normalization}
 \gcd(r,s,m)=1.
\end{equation}
Put
\[
 \cR=\{r+jm:j\geq0\},\qquad
 \cS=\{s+jm:j\geq0\}.
\]
For a partition $\lambda$, let $N_A(\lambda)$ be the number of parts of
$\lambda$ that belong to $A$, counted with multiplicity, and define
\[
 D(\lambda)=N_{\cR}(\lambda)-N_{\cS}(\lambda).
\]
Write $\mathcal P(n)$ for the set of partitions of $n$, $p(n)=|\mathcal
P(n)|$, and
\[
 p_{\cR>\cS}(n)=|\{\lambda\in\mathcal P(n):D(\lambda)>0\}|.
\]
For $a,b>0$, let
\[
 I_x(a,b)=\frac{1}{B(a,b)}\int_0^x u^{a-1}(1-u)^{b-1}\,du
\]
be the regularized incomplete beta function, where
\[
 B(a,b)=\int_0^1u^{a-1}(1-u)^{b-1}\,du
        =\frac{\Gamma(a)\Gamma(b)}{\Gamma(a+b)}
\]
is Euler's beta function and beta--gamma identity; see
\cite[Section 5.12]{Olver2010}.  Also,
$\Gamma(c)=\int_0^\infty x^{c-1}\e^{-x}\,dx$ is Euler's gamma function.
Equivalently, if $U$ has the $\operatorname{Beta}(a,b)$ distribution with
probability density
\[
 f_U(u)=\frac{u^{a-1}(1-u)^{b-1}}{B(a,b)},\qquad 0<u<1,
\]
then $I_x(a,b)=\PP(U\leq x)$.

For $c>0$, we write $G_c\sim\operatorname{Gamma}(c,1)$ if $G_c$ is a
positive random variable with density
\begin{equation}\label{eq:gamma-random-variable}
 g_c(x)=\frac{x^{c-1}\e^{-x}}{\Gamma(c)},\qquad x>0.
\end{equation}
\begin{theorem}
\label{thm:main}
Let \eqref{eq:distinct-residues} and
\eqref{eq:coprime-normalization} hold, and set $a=r/m$ and $b=s/m$.
For $t>0$, define
\begin{equation}\label{eq:mu-def}
 \mu(t)=\sum_{k\geq1}\frac{k}{\e^{tk}-1},
\end{equation}
and let $t_n$ be the unique positive solution of $\mu(t_n)=n$.  If
$\Lambda_n$ is uniformly distributed on $\mathcal P(n)$---that is, each of
the $p(n)$ partitions is selected with probability $1/p(n)$---then
\begin{equation}\label{eq:distribution-main}
 t_nD(\Lambda_n)\ \xrightarrow{\ d\ }\
 W:=\frac{\log G_b-\log G_a}{m},
\end{equation}
where $G_a$ and $G_b$ are independent gamma random variables with respective
shape parameters $a$ and $b$ and common rate $1$, in the notation above.
Consequently,
\begin{equation}\label{eq:bias-main}
 \lim_{n\to\infty}\frac{p_{\cR>\cS}(n)}{p(n)}
 =I_{1/2}(a,b).
\end{equation}
Moreover,
\begin{equation}\label{eq:tn-asymp-intro}
 t_n\sim\frac{\pi}{\sqrt{6n}}.
\end{equation}
\end{theorem}

The limit in \eqref{eq:distribution-main} is generally non-Gaussian, that
is, it is not a normal distribution.  In fact, the change of variables
$(G_a,G_b)=(v,vq)$, whose Jacobian is $v$, shows that the ratio
$Q=G_b/G_a$ has the beta-prime density
$x^{b-1}/(B(a,b)(1+x)^{a+b})$ for $x>0$.  Since
 $Q=\e^{mW}$, the substitution $x=\e^{mw}$ gives the continuous density
\begin{equation}\label{eq:W-density}
 f_W(w)=\frac{m\,\e^{bmw}}
 {B(a,b)(1+\e^{mw})^{a+b}},\qquad w\in\mathbb R.
\end{equation}
Since $t_n$ is asymptotic to a constant times $n^{-1/2}$, the difference of
the two multiplicity counts typically has order $\sqrt n$.  Formula
\eqref{eq:W-density} shows more: its limiting shape remembers both starting
points $r$ and $s$, not merely their order.  In the special case $s=m$,
formula \eqref{eq:bias-main} becomes $2^{-r/m}$, the family of constants
predicted by Li and Zhao; when $r=1$, this includes the previously known
residue-class bias.

The proof has two halves.  The first evaluates Li and Zhao's finite-$N$
integral by rescaling and dominated convergence; it produces
$I_{1/2}(a,b)$.  The second treats unrestricted partitions with the
conditioning device of Fristedt \cite{Fristedt1993}: multiplicities become
independent geometric variables, conditioning on the total size gives the
uniform measure on $\mathcal P(n)$, and a $q$-gamma limit identifies the
characteristic function in \eqref{eq:distribution-main}.  Since we
condition on a single lattice value, a conditional local limit estimate
proved by Fourier inversion is needed; a Taylor expansion controls the
central arc and a trigonometric-sum estimate makes the complementary arcs
negligible.

Section~\ref{sec:truncation} evaluates the finite-truncation limit.
Section~\ref{sec:boltzmann} introduces the Boltzmann model and proves the
unconditioned limit.  The conditioning argument is given in
Section~\ref{sec:conditioning}, and Section~\ref{sec:completion} completes
the proof of the conjecture.  Section~\ref{sec:numerics} treats three
explicit constants, and Section~\ref{sec:further} records consequences.

\section{The finite-truncation limit}
\label{sec:truncation}

We first calculate the iterated limit in which the size of the partition
tends to infinity before the truncation is removed.  This part is analytic
and independent of the probabilistic argument used later.  Besides proving
one half of the desired commutation identity, the calculation already
reveals the beta distribution that will reappear from the random-partition
model.

\subsection{The truncated model}

For $N\geq1$, set
\[
 \cR_N=\{r+jm:0\leq j<N\},\qquad
 \cS_N=\{s+jm:0\leq j<N\},
\]
and
\[
 M_N=\max(r,s)+m(N-1),\qquad
 \mathcal I_N=\{1,\ldots,M_N\}\setminus(\cR_N\cup\cS_N).
\]
Let $C_{n,N}$ be the proportion of partitions of $n$, using only parts in
$\{1,\ldots,M_N\}$, for which the number of parts in $\cR_N$ exceeds the
number in $\cS_N$.  This agrees with Li and Zhao's notation because
$\cR_N\cup\cS_N\cup\mathcal I_N=\{1,\ldots,M_N\}$.

For each fixed $n$, all parts of a partition of $n$ are at most $n$.  Choose
$N$ so large that
\[
 r+m(N-1)\geq n\qquad\text{and}\qquad s+m(N-1)\geq n.
\]
Then $M_N\geq n$, and $\cR_N$ and $\cS_N$ contain every member of $\cR$
and $\cS$, respectively, that can occur in a partition of $n$.  At that
point both the numerator and denominator defining $C_{n,N}$ have stabilized.
Therefore
\begin{equation}\label{eq:truncate-first}
 \lim_{N\to\infty}C_{n,N}=\frac{p_{\cR>\cS}(n)}{p(n)}.
\end{equation}

The subtle issue is that the threshold at which this stabilization occurs
depends on $n$.  Thus \eqref{eq:truncate-first} alone does not permit the
limits in $n$ and $N$ to be interchanged.

\begin{proposition}\label{prop:finite-limit}
Under \eqref{eq:distinct-residues} and
\eqref{eq:coprime-normalization}, with $a=r/m$ and $b=s/m$,
\begin{equation}\label{eq:finite-iterated-limit}
 \lim_{N\to\infty}\lim_{n\to\infty}C_{n,N}=I_{1/2}(a,b).
\end{equation}
\end{proposition}

\subsection{Asymptotics of the integral}

\begin{proof}
Formula (18) of Li and Zhao \cite{LiZhao2026}, rewritten in gamma notation,
gives the following exact expression for the fixed-$N$ limit:
\begin{equation}\label{eq:Li-Zhao-formula}
 L_N:=\lim_{n\to\infty}C_{n,N}
 =\frac{\Gamma(N+a)\Gamma(N+b)}
 {\Gamma(a)\Gamma(b)\Gamma(N)^2}\,J_N,
\end{equation}
where
\begin{equation}\label{eq:JN}
 J_N=\int_0^1 x^{b-1}(1-x)^{N-1}
       \left(\int_0^x y^{a-1}(1-y)^{N-1}\,dy\right)dx.
\end{equation}
The cited formula is stated in terms of the finite allowed-parts model
$\cR_N\cup\cS_N\cup\mathcal I_N$.  In the present notation this is exactly
the model using all parts in $\{1,\ldots,M_N\}$.  For $N\geq2$, the
normalization \eqref{eq:coprime-normalization} gives
$\gcd(\cR_N\cup\cS_N)=1$, which is the arithmetic hypothesis needed in the
cited finite-set theorem.  The parameters of the marked progressions enter
the resulting homogeneous volume ratio through $a=r/m$ and $b=s/m$.
The factors $(1-x)^{N-1}$ and $(1-y)^{N-1}$ force the main contribution to
come from $x,y$ of order $N^{-1}$.  We make the substitutions
$x=u/N$ and $y=v/N$.  Since
\[
 dx\,dy=N^{-2}\,du\,dv,\qquad
 x^{b-1}y^{a-1}=N^{-(a+b-2)}u^{b-1}v^{a-1}.
\]
the Jacobian and the powers in the integrand contribute $N^{-(a+b)}$.
Therefore
\begin{align}
 N^{a+b}J_N
 ={}&\int_0^N u^{b-1}\left(1-\frac{u}{N}\right)^{N-1}
 \left(\int_0^u v^{a-1}
       \left(1-\frac{v}{N}\right)^{N-1}\,dv\right)du.
 \label{eq:JN-scaled}
\end{align}
For each fixed pair $0\leq v\leq u$, we have
\[
 \left(1-\frac{u}{N}\right)^{N-1}\longrightarrow\e^{-u},
 \qquad
 \left(1-\frac{v}{N}\right)^{N-1}\longrightarrow\e^{-v}.
\]
To pass this pointwise limit through both integrals, we need a bound that is
independent of $N$.  Since $\log(1-z)\leq-z$ for $0\leq z<1$, for
$N\geq2$ and $0\leq u<N$,
\[
 \left(1-\frac{u}{N}\right)^{N-1}
 \leq\exp\left(-\frac{N-1}{N}u\right)
 \leq \e^{-u/2},
\]
and the endpoint $u=N$ follows by continuity.  The same bound holds with
$v$ in place of $u$.  After extending the integrand in
\eqref{eq:JN-scaled} by zero to $(N,\infty)$, it is bounded by
\[
 u^{b-1}\e^{-u/2}\int_0^u v^{a-1}\e^{-v/2}\,dv.
\]
This is integrable on $(0,\infty)$.  Near zero, the inner integral is
$O(u^a)$, so the whole expression is $O(u^{a+b-1})$; at infinity, the inner
integral is bounded by $2^a\Gamma(a)$ and the factor $\e^{-u/2}$ gives
exponential decay.  The dominated convergence theorem therefore applies
and gives
\begin{equation}\label{eq:JN-limit}
 \lim_{N\to\infty}N^{a+b}J_N
 =\int_0^\infty u^{b-1}\e^{-u}
   \left(\int_0^u v^{a-1}\e^{-v}\,dv\right)du.
\end{equation}
For fixed $c>0$, the only form of Stirling's formula needed here is the
gamma-ratio estimate
\[
 \frac{\Gamma(N+c)}{\Gamma(N)N^c}\longrightarrow1.
\]
This is the fixed-shift consequence of Stirling's formula; see
\cite[Section 5.11]{Olver2010}.  Equivalently, replacing $\Gamma(N+c)$ by
$\Gamma(N)N^c$ incurs a relative error tending to zero.
Thus the gamma prefactor in \eqref{eq:Li-Zhao-formula} is asymptotic to
$N^{a+b}/(\Gamma(a)\Gamma(b))$.  Combining this with
\eqref{eq:JN-limit} yields
\begin{equation}\label{eq:gamma-probability-integral}
 \lim_{N\to\infty}L_N
 =\frac{1}{\Gamma(a)\Gamma(b)}
 \int_0^\infty u^{b-1}\e^{-u}
 \left(\int_0^u v^{a-1}\e^{-v}\,dv\right)du.
\end{equation}

\medskip
\noindent\emph{Identification of the limiting integral.}
Let $G_a$ and $G_b$ be independent random variables with the gamma
distributions in \eqref{eq:gamma-random-variable}.  Their densities are
\[
 \frac{v^{a-1}\e^{-v}}{\Gamma(a)}
 \quad\text{and}\quad
 \frac{u^{b-1}\e^{-u}}{\Gamma(b)},\qquad u,v>0.
\]
Integrating their joint density over the region $0<v<u<\infty$ shows
directly that the right-hand side of
\eqref{eq:gamma-probability-integral} is $\PP(G_a<G_b)$.

Indeed, the inner integral is the density of $G_a$ integrated over
$0<v<u$, while the outer factor is the density of $G_b$ at $u$.  Thus the
double integral is
$\PP(G_a<G_b)=\int_0^\infty f_{G_b}(u)\PP(G_a<u)\,du$.

To identify this probability, put
\[
 X=\frac{G_a}{G_a+G_b},\qquad Y=G_a+G_b.
\]
The inverse transformation is $(G_a,G_b)=(YX,Y(1-X))$ and has Jacobian
$Y$.  Hence the joint density of $(X,Y)$ factors as
\[
 \frac{x^{a-1}(1-x)^{b-1}}{B(a,b)}
 \cdot
 \frac{y^{a+b-1}\e^{-y}}{\Gamma(a+b)},
 \qquad 0<x<1, y>0.
\]
Thus $X$ and $Y$ are independent, with
$X\sim\operatorname{Beta}(a,b)$, and in particular
\[
 \frac{G_a}{G_a+G_b}\sim\operatorname{Beta}(a,b),
\]
and hence
\[
 \PP(G_a<G_b)
 =\PP\left(\frac{G_a}{G_a+G_b}<\frac12\right)
 =I_{1/2}(a,b).
\]
This proves \eqref{eq:finite-iterated-limit}.
\end{proof}

\section{The Boltzmann model and the unconditioned limit}
\label{sec:boltzmann}

\subsection{Independent multiplicities and exact conditioning}

We begin by embedding the uniform measure on partitions of a fixed integer
into a one-parameter family of product measures.  This is often called the
Boltzmann measure, or the conditioning device for partitions.  A probability
measure on $\mathcal P$ assigns nonnegative weights, of total mass $1$, to
the possible partitions.  In the Boltzmann model the weight decreases
exponentially with the size; the auxiliary parameter $t$ controls how
strongly large partitions are suppressed.  Fix $t>0$, write
$\rho_k=\e^{-tk}$, and, independently for each $k\geq1$, let $X_k$ have the
geometric distribution
\begin{equation}\label{eq:geometric}
 \PP_t(X_k=j)=(1-\rho_k)\rho_k^j=(1-\e^{-tk})\e^{-tkj},
 \qquad j=0,1,2,\ldots.
\end{equation}
Thus $X_k$ counts how many copies of the part $k$ occur: zero copies have
probability $1-\rho_k$, and every additional copy contributes another
factor $\rho_k$.  Independence means, more precisely, that for any finite
set of indices the probability of prescribed values is the product of the
individual probabilities.  We regard $X_k$ as the multiplicity of the part
$k$.  Thus every realization
$(X_1,X_2,\ldots)$ determines the formal partition having $X_k$ copies of
$k$.  Its total size and the difference between the two multiplicity counts
are respectively
\begin{equation}\label{eq:T-and-D}
 T=\sum_{k\geq1}kX_k,
 \qquad
 D=\sum_{j\geq0}X_{r+jm}-\sum_{j\geq0}X_{s+jm}.
\end{equation}
These random series are well defined.  Indeed,
here $\EE_tY$ denotes the expectation, or probability-weighted mean, of a
random variable $Y$ under $\PP_t$, and
\[
 \EE_tT=\sum_{k\geq1}\frac{k}{\e^{tk}-1}<\infty
\]
for every fixed $t>0$, since the summand is $O(k\e^{-tk})$ as
$k\to\infty$.  A nonnegative random variable with finite expectation is
finite almost surely, where ``almost surely'' means with probability one.
Hence $T<\infty$ almost surely; because
$\sum_{k\geq1}X_k\leq T$, only finitely many parts occur and both sums
defining $D$ are also almost surely finite.

Let $\mathcal P$ denote the set of all integer partitions, including the
empty partition, and write $|\lambda|$ for the size of $\lambda$.  The
normalizing constant of the product measure is Euler's partition product
\[
 Z(t)=\sum_{\lambda\in\mathcal P}\e^{-t|\lambda|}
     =\prod_{k\geq1}\left(\sum_{j\geq0}\e^{-tkj}\right)
     =\prod_{k\geq1}(1-\e^{-tk})^{-1}.
\]
The product is finite for $t>0$: for instance,
$-\log(1-\e^{-tk})=O(\e^{-tk})$ as $k\to\infty$.  Independence in
\eqref{eq:geometric} now shows that $(X_k)$ is precisely the multiplicity
sequence of a random partition with Boltzmann law
\begin{equation}\label{eq:Boltzmann-measure}
 \PP_t(\lambda)=\frac{\e^{-t|\lambda|}}{Z(t)},
 \qquad \lambda\in\mathcal P.
\end{equation}
Here ``law'' is synonymous with ``probability distribution.''  Formula
\eqref{eq:Boltzmann-measure} makes the preceding description precise:
partitions of the same size have equal weight, while increasing the size by
one multiplies the unnormalized weight by $\e^{-t}$.

The parameter $t$ biases the measure toward smaller or larger total size,
but this bias disappears completely once the size is fixed.  To see this
explicitly, let $\lambda$ be a partition of $n$ with multiplicity sequence
$(x_k)_{k\geq1}$, so that $\sum_{k\geq1}kx_k=n$.  Then
\[
 \begin{aligned}
 \PP_t(X_k=x_k\text{ for all }k)
 &=\prod_{k\geq1}(1-\e^{-tk})\e^{-tkx_k}\\
 &=Z(t)^{-1}\exp\left(-t\sum_{k\geq1}kx_k\right)
 =Z(t)^{-1}\e^{-tn}.
 \end{aligned}
\]
The last expression depends on $\lambda$ only through $n$.  Since there are
$p(n)$ partitions of $n$, summing this identity over $\mathcal P(n)$ gives
\begin{equation}\label{eq:size-probability}
 \PP_t(T=n)=p(n)Z(t)^{-1}\e^{-tn}.
\end{equation}
It follows, for each $\lambda\in\mathcal P(n)$, that
\[
 \PP_t\bigl((X_k)_{k\geq1}\text{ represents }\lambda\mid T=n\bigr)
 =\frac{Z(t)^{-1}\e^{-tn}}
        {p(n)Z(t)^{-1}\e^{-tn}}
 =\frac1{p(n)}.
\]
For a random variable or random vector $Y$, write $\mathcal L_t(Y)$ for its
probability distribution under $\PP_t$, and write $\mathcal L_t(Y\mid A)$ for
its conditional distribution under $\PP_t$ given an event $A$ of positive
probability.  In other words, $\mathcal L_t(Y\mid A)$ is the probability
distribution of $Y$ after restricting to $A$ and renormalizing.  Let
$\mathcal U_n$ denote the uniform probability measure on $\mathcal P(n)$,
so that $\mathcal U_n(\{\lambda\})=1/p(n)$ for every $\lambda\in\mathcal P(n)$.
Under the natural identification of a partition with its multiplicity
sequence, the essential conditioning relation is the exact identity
\begin{equation}\label{eq:conditioning-identity}
 \mathcal L_t((X_k)_{k\geq1}\mid T=n)
 =\mathcal U_n.
\end{equation}
For example, conditioning on $T=n$ means dividing
$\PP_t(\,\cdot\,,T=n)$ by $\PP_t(T=n)$; it discards all outcomes of other
sizes and rescales the remaining probabilities to have total mass $1$.
In particular, the conditional law in \eqref{eq:conditioning-identity} does
not depend on $t$.  We shall exploit this freedom by choosing $t=t_n$ so
that $\EE_{t_n}T=n$.  With this saddle-point choice, the event on which we
condition is located at the center, rather than in a tail, of the
unconditioned distribution of $T$.

For example, consider the partition $7=4+2+1$.  Its multiplicities are
$x_1=x_2=x_4=1$ and all other $x_k=0$, and its probability is
$\prod_k(1-e^{-tk})e^{-tkx_k}=Z(t)^{-1}e^{-7t}$.  Every partition of $7$
has the same probability.  Hence, after conditioning on $T=7$, each one
has probability $1/p(7)$.

\subsection{The saddle-point scale}

We next record the scale of that distribution.  For a geometric random
variable with parameter $\rho_k$ as in \eqref{eq:geometric},
\begin{equation}\label{eq:geometric-moments}
 \EE_tX_k=\frac{\rho_k}{1-\rho_k}=\frac1{\e^{tk}-1},
 \qquad
 \operatorname{Var}_t(X_k)=\frac{\rho_k}{(1-\rho_k)^2}.
\end{equation}
Here
$\operatorname{Var}_t(Y)=\EE_t[(Y-\EE_tY)^2]$ is the variance, which
measures squared fluctuation about the mean; its square root is the standard
deviation.

\begin{lemma}\label{lem:moments}
As $t\downarrow0$,
\begin{align}
 \mu(t):=\EE_tT
 &=\sum_{k\geq1}\frac{k}{\e^{tk}-1}
 =\frac{\pi^2}{6t^2}+O(t^{-1}),
 \label{eq:mean-asymp}\\
 \sigma^2(t):=\operatorname{Var}_t(T)
 &=\sum_{k\geq1}\frac{k^2\e^{-tk}}{(1-\e^{-tk})^2}
 =\frac{\pi^2}{3t^3}+O(t^{-2}).
 \label{eq:variance-asymp}
\end{align}
In particular, the equation $\mu(t_n)=n$ has a unique solution and
$t_n\sim\pi/\sqrt{6n}$.
\end{lemma}

\begin{proof}
By independence and \eqref{eq:geometric-moments}, Tonelli's theorem gives
\[
 \EE_tT=\sum_{k\geq1}k\EE_tX_k
       =\sum_{k\geq1}\frac{k}{\e^{tk}-1}.
\]
Since independent variables have additive variances, and the variance
series converges absolutely for fixed $t>0$,
\[
 \operatorname{Var}_t(T)
 =\sum_{k\geq1}k^2\operatorname{Var}_t(X_k)
 =\sum_{k\geq1}\frac{k^2\e^{-tk}}{(1-\e^{-tk})^2}.
\]

It remains to estimate the two series.  Introduce
\[
 f_1(x)=\frac{x}{\e^x-1},\qquad
 f_2(x)=\frac{x^2\e^{-x}}{(1-\e^{-x})^2},
\]
with the continuous extensions $f_1(0)=f_2(0)=1$.  Then
\begin{equation}\label{eq:moment-Riemann-sums}
 t^2\mu(t)=t\sum_{k\geq1}f_1(tk),\qquad
 t^3\sigma^2(t)=t\sum_{k\geq1}f_2(tk).
\end{equation}
Both functions are continuously differentiable on $[0,\infty)$, and
$f_j,f_j'\in L^1(0,\infty)$.  For any such function $f$, comparison on
each interval $[(k-1)t,kt]$ gives the quantitative Riemann-sum estimate
\begin{align*}
 \left|t\sum_{k\geq1}f(kt)-\int_0^\infty f(x)\,dx\right|
 &\leq\sum_{k\geq1}\int_{(k-1)t}^{kt}|f(kt)-f(x)|\,dx\\
 &\leq t\int_0^\infty|f'(x)|\,dx=O(t).
\end{align*}
Applying this to \eqref{eq:moment-Riemann-sums} yields
\begin{align*}
 t^2\mu(t)
 &=\int_0^\infty\frac{x}{\e^x-1}\,dx+O(t),\\
 t^3\sigma^2(t)
 &=\int_0^\infty
 \frac{x^2\e^{-x}}{(1-\e^{-x})^2}\,dx+O(t).
\end{align*}
Finally, the nonnegative expansions
\[
 \frac1{\e^x-1}=\sum_{\ell\geq1}\e^{-\ell x},
 \qquad
 \frac{\e^{-x}}{(1-\e^{-x})^2}
 =\sum_{\ell\geq1}\ell\e^{-\ell x}
\]
allow termwise integration and give
\[
 \int_0^\infty\frac{x}{\e^x-1}\,dx
 =\sum_{\ell\geq1}\frac1{\ell^2}=\frac{\pi^2}{6},
 \qquad
 \int_0^\infty\frac{x^2\e^{-x}}{(1-\e^{-x})^2}\,dx
 =2\sum_{\ell\geq1}\frac1{\ell^2}=\frac{\pi^2}{3}.
\]
Indeed, $\int_0^\infty x\e^{-\ell x}\,dx=1/\ell^2$ and
$\int_0^\infty x^2\e^{-\ell x}\,dx=2/\ell^3$.  Therefore the constants
$\pi^2/6$ and $\pi^2/3$ follow from
$\sum_{\ell\geq1}\ell^{-2}=\pi^2/6$.
This proves \eqref{eq:mean-asymp} and \eqref{eq:variance-asymp}.

For completeness, the series defining $\mu$ converges locally uniformly on
$(0,\infty)$, and every summand is continuous and strictly decreasing in
$t$.  Thus $\mu$ is continuous and strictly decreasing.  The asymptotic
formula shows that $\mu(t)\to\infty$ as $t\downarrow0$, while elementary
domination by a convergent geometric series shows that $\mu(t)\to0$ as
$t\to\infty$.  Hence $\mu(t_n)=n$ has a unique solution.  Substitution in
\eqref{eq:mean-asymp} gives
$n=\pi^2/(6t_n^2)(1+O(t_n))$, and therefore
$t_n\sim\pi/\sqrt{6n}$.
\end{proof}

\subsection{The \texorpdfstring{$q$}{q}-gamma evaluation}

The distribution of $tD$ under the unconditioned measure has a particularly
simple limit.  The appearance of the $q$-gamma function is natural because
each residue class contributes a product whose exponents form an arithmetic
progression.  Recall that the characteristic function of a real random
variable $Y$ is
\[
 \chi_Y(\theta)=\EE(\e^{i\theta Y}),\qquad \theta\in\mathbb R.
\]
It is always well defined because $|\e^{i\theta Y}|=1$, and it uniquely
determines the distribution of $Y$.  We use the standard notation
\[
 (z;q)_\infty=\prod_{j\geq0}(1-zq^j),\qquad 0<q<1.
\]
For every fixed complex $z$, this product converges because
$\sum_{j\geq0}|z|q^j<\infty$ (it is zero when one of its factors is zero).
For complex exponents we use
$q^z=\exp(z\log q)$, where $\log q<0$ is the real logarithm.

\begin{proposition}\label{prop:unconditioned}
As $t\downarrow0$,
\begin{equation}\label{eq:unconditioned-convergence}
 tD\xrightarrow{\ d\ }\frac{\log G_b-\log G_a}{m}.
\end{equation}
Equivalently, for every real $\theta$,
\begin{equation}\label{eq:phi-limit}
 \Phi_t(\theta):=\EE_t\e^{i\theta tD}
 \longrightarrow
 \phi(\theta):=
 \frac{\Gamma(a-i\theta/m)\Gamma(b+i\theta/m)}
 {\Gamma(a)\Gamma(b)}.
\end{equation}
The convergence is locally uniform in $\theta\in\mathbb R$.
\end{proposition}

\begin{proof}
We first calculate the contribution of one multiplicity.  If $X_k$ has the
distribution \eqref{eq:geometric}, then its probability generating function,
defined for an integer-valued nonnegative variable $X$ by
$G_X(z)=\EE(z^X)=\sum_{j\geq0}\PP(X=j)z^j$, is
\begin{equation}\label{eq:single-geometric-pgf}
 \EE_t z^{X_k}
 =(1-\e^{-tk})\sum_{\ell\geq0}(z\e^{-tk})^\ell
 =\frac{1-\e^{-tk}}{1-z\e^{-tk}},
 \qquad |z|<\e^{tk}.
\end{equation}
In the statistic $D$, a part from $\cR$ carries sign $+1$, whereas a part
from $\cS$ carries sign $-1$.  We therefore substitute
$z=\e^{i\theta t}$ for $k\in\cR$ and
$z=\e^{-i\theta t}$ for $k\in\cS$.  Independence and
\eqref{eq:single-geometric-pgf} give
\begin{align}
 \Phi_t(\theta)
 ={}&\prod_{j\geq0}
 \frac{1-\e^{-t(r+jm)}}
      {1-\e^{-t(r+jm-i\theta)}}
 \prod_{j\geq0}
 \frac{1-\e^{-t(s+jm)}}
      {1-\e^{-t(s+jm+i\theta)}}.
 \label{eq:Phi-product}
\end{align}
For example,
\[
 1-\e^{-t(r+jm)}\e^{i\theta t}
 =1-\e^{-t(r+jm-i\theta)};
\]
the opposite sign for $\cS$ accounts for $s+jm+i\theta$ in the second
denominator.  Absolute convergence follows from
$\sum_{j\geq0}\e^{-t(r+jm)}<\infty$ and its analogue for $s$.

Now set $q=\e^{-mt}$.  Since $a=r/m$ and $b=s/m$,
\[
 \e^{-t(r+jm)}=q^{a+j},\qquad
 \e^{-t(r+jm-i\theta)}=q^{a+j-i\theta/m},
\]
and similarly for the progression beginning at $s$.  Grouping the factors
according to the definition of the $q$-Pochhammer symbol yields
\begin{equation}\label{eq:Phi-q-product}
 \Phi_t(\theta)=
 \frac{(q^a;q)_\infty}{(q^{a-i\theta/m};q)_\infty}
 \frac{(q^b;q)_\infty}{(q^{b+i\theta/m};q)_\infty}.
\end{equation}

We next recall the $q$-gamma identity
\begin{equation}\label{eq:q-gamma}
 \Gamma_q(z)=(1-q)^{1-z}
 \frac{(q;q)_\infty}{(q^z;q)_\infty}
\end{equation}
and the convergence
\[
 \Gamma_q(z)\longrightarrow\Gamma(z)\qquad(q\uparrow1),
\]
which is locally uniform for $\operatorname{Re}z>0$; see, for example,
\cite[Chapter 1]{GasperRahman2004}.  Solving \eqref{eq:q-gamma} for the
$q$-Pochhammer symbol gives
\[
 (q^z;q)_\infty
 =(1-q)^{1-z}\frac{(q;q)_\infty}{\Gamma_q(z)}.
\]
Consequently, for $x$ and $y$ in a compact subset of the right half-plane,
\begin{equation}\label{eq:q-ratio}
 \frac{(q^x;q)_\infty}{(q^y;q)_\infty}
 =(1-q)^{y-x}\frac{\Gamma_q(y)}{\Gamma_q(x)}
 \sim(1-q)^{y-x}\frac{\Gamma(y)}{\Gamma(x)}.
\end{equation}
Substituting the definition of $\Gamma_q$ twice and cancelling the common
factor $(q;q)_\infty$ gives \eqref{eq:q-ratio}.  Since $a,b>0$, the numbers
$a-i\theta/m$ and $b+i\theta/m$ have positive real parts for real
$\theta$, so the quoted $q$-gamma convergence applies.
Apply \eqref{eq:q-ratio} first with
$(x,y)=(a,a-i\theta/m)$ and then with
$(x,y)=(b,b+i\theta/m)$.  The first ratio contributes the power
$(1-q)^{-i\theta/m}$, and the second contributes
$(1-q)^{i\theta/m}$.  These powers cancel exactly.  We obtain
\[
 \Phi_t(\theta)\longrightarrow
 \frac{\Gamma(a-i\theta/m)}{\Gamma(a)}
 \frac{\Gamma(b+i\theta/m)}{\Gamma(b)}
 =\phi(\theta).
\]
Because the $\Gamma_q$ convergence is locally uniform, so is this
convergence for $\theta$ in compact subsets of $\mathbb R$.

By \eqref{eq:gamma-random-variable}, for $c>0$ and $u\in\mathbb R$,
\[
 \EE G_c^{iu}
 =\frac1{\Gamma(c)}\int_0^\infty x^{c+iu-1}\e^{-x}\,dx
 =\frac{\Gamma(c+iu)}{\Gamma(c)}.
\]
Using independence of $G_a$ and $G_b$, the characteristic function of
$W=m^{-1}(\log G_b-\log G_a)$ is therefore
\[
 \EE\e^{i\theta W}
 =\EE G_b^{i\theta/m}\,\EE G_a^{-i\theta/m}
 =\frac{\Gamma(a-i\theta/m)\Gamma(b+i\theta/m)}
 {\Gamma(a)\Gamma(b)}
 =\phi(\theta).
\]
In particular, $\phi$ is continuous at the origin and $\phi(0)=1$.
We use Levy's continuity theorem in the following form (see
\cite[Chapters 2--3]{Billingsley1999}): if $\chi_n$ are
characteristic functions and $\chi_n(\theta)\to\chi(\theta)$ pointwise,
where $\chi$ is continuous at the origin, then $\chi$ is a characteristic
function and the corresponding random variables converge in distribution.
Applying this theorem to $\Phi_t$ gives
\eqref{eq:unconditioned-convergence}.
\end{proof}

\section{Conditioning on the total size}
\label{sec:conditioning}

Proposition~\ref{prop:unconditioned} concerns a random partition whose total
size fluctuates.  We now show that the same limit holds after conditioning
on $T=n$, provided $t=t_n$.  This is the only place where a local, rather
than global, limit estimate is required.  The distinction matters: weak
convergence (another name for convergence in distribution) describes
probabilities of intervals on the fluctuation scale, whereas the
conditioning event $\{T=n\}$ is a single lattice point whose probability
tends to zero.  A local limit estimate supplies the finer approximation to
this individual point probability.

\subsection{Fourier formulation of the conditioning problem}

The purpose of this subsection is to express the conditioning operation in
terms of Fourier integrals.  Under the Boltzmann measure, the two quantities
of interest are random:
\[
 T=\sum_{k\geq1}kX_k \quad\text{and}\quad
 D=\sum_{k\geq1}\varepsilon_kX_k.
\]
The variable $T$ records the total size of the partition, whereas $D$
records the excess of parts from $\cR$ over parts from $\cS$.  The uniform
measure on partitions of $n$ is obtained by restricting the Boltzmann model
to the single event $T=n$.  Consequently, the problem is not to study $D$
alone, but to study $D$ and $T$ simultaneously and then extract the part
corresponding to $T=n$.

Put
\begin{equation}\label{eq:epsilon}
 \varepsilon_k=\1_{\{k\in\cR\}}-\1_{\{k\in\cS\}}.
\end{equation}
Thus $\varepsilon_k=1$ when $k\in\cR$, $\varepsilon_k=-1$ when
$k\in\cS$, and $\varepsilon_k=0$ otherwise.  The assumption
$r\not\equiv s\pmod m$ ensures that $\cR$ and $\cS$ are disjoint, so
$\varepsilon_k\in\{-1,0,1\}$ is unambiguous.  For each realization of the
multiplicities,
\begin{equation}\label{eq:phase-decomposition}
 \theta tD+\varphi T
 =\sum_{k\geq1}(\theta t\varepsilon_k+k\varphi)X_k.
\end{equation}
For example, a copy of a part $k\in\cR$ contributes the phase
$\theta t+k\varphi$, a copy of a part $k\in\cS$ contributes
$-\theta t+k\varphi$, and an unmarked part contributes only $k\varphi$.
This explains why the two Fourier variables have different roles: $\theta$
tracks the statistic $D$, while $\varphi$ tracks the integer-valued size
$T$.

For real $\theta,\varphi$, define the joint characteristic function
\begin{equation}\label{eq:J-def}
 J_t(\theta,\varphi)
 =\EE_t\exp(i\theta tD+i\varphi T).
\end{equation}
Here $\theta$ is the Fourier variable corresponding to $tD$, while
$\varphi$ is the Fourier variable corresponding to the total size $T$.
This is the characteristic function of the two-dimensional random vector
$(tD,T)$.  Setting one variable equal to zero gives the characteristic
function of the other coordinate:
\begin{equation}\label{eq:J-marginals}
 J_t(\theta,0)=\Phi_t(\theta),
 \qquad
 J_t(0,\varphi)=\EE_t\e^{i\varphi T}.
\end{equation}
The first equality connects this subsection with the $q$-gamma calculation
in Proposition~\ref{prop:unconditioned}.  The second one is the
characteristic function to be used for the local limit estimate of $T$.

We now apply Fourier inversion to the integer-valued variable $T$.  Since
$T$ takes values in $\{0,1,2,\ldots\}$, we may write
\begin{equation}\label{eq:J-series}
 J_t(\theta,\varphi)
 =\sum_{\ell\geq0}
 \EE_t\bigl(\e^{i\theta tD};T=\ell\bigr)\e^{i\ell\varphi}.
\end{equation}
The series is absolutely convergent, since the sum of the absolute values of
its coefficients is at most $1$.  The elementary orthogonality relation
\begin{equation}\label{eq:Fourier-orthogonality}
 \frac1{2\pi}\int_{-\pi}^{\pi}
 \e^{i(\ell-n)\varphi}\,d\varphi
 =\begin{cases}1,&\ell=n,\\0,&\ell\ne n,
 \end{cases}
\end{equation}
follows by integrating the exponential explicitly.  Thus multiplying
\eqref{eq:J-series} by $\e^{-in\varphi}$ and integrating over one period
selects exactly the coefficient with $\ell=n$.  This gives the lattice
Fourier inversion formulas
\begin{equation}\label{eq:Fourier-roadmap}
 \begin{aligned}
 \EE_t(\e^{i\theta tD};T=n)
 &=\frac1{2\pi}\int_{-\pi}^{\pi}
   \e^{-in\varphi}J_t(\theta,\varphi)\,d\varphi,\\
 \PP_t(T=n)
 &=\frac1{2\pi}\int_{-\pi}^{\pi}
   \e^{-in\varphi}J_t(0,\varphi)\,d\varphi.
 \end{aligned}
\end{equation}
The lattice span is one because the Boltzmann model contains the part $1$;
there is therefore no additional periodicity factor in this inversion.

Here $\EE_t(Y;T=n)$ is shorthand for
$\EE_t[Y\1_{\{T=n\}}]$.  Therefore the conditional characteristic function
of $tD$ is
\begin{equation}\label{eq:conditional-cf-ratio}
 \EE_t(\e^{i\theta tD}\mid T=n)
 =\frac{\EE_t(\e^{i\theta tD};T=n)}{\PP_t(T=n)}
 =\frac{\displaystyle\int_{-\pi}^{\pi}
     \e^{-in\varphi}J_t(\theta,\varphi)\,d\varphi}
    {\displaystyle\int_{-\pi}^{\pi}
     \e^{-in\varphi}J_t(0,\varphi)\,d\varphi}.
\end{equation}
The factors $1/(2\pi)$ cancel in the last ratio.  Consequently, the goal
of the rest of Section~\ref{sec:conditioning} is to prove that the numerator
in \eqref{eq:conditional-cf-ratio} is asymptotic to
$\Phi_t(\theta)$ times the denominator.

The reason for this asymptotic relation is that, on the central Fourier arc
which gives the main contribution to the inversion integral, the joint
characteristic function approximately factors as
\[
 J_t(\theta,\varphi)
 =
 \Phi_t(\theta)J_t(0,\varphi)(1+o(1)).
\]
Thus the statistic $tD$ is asymptotically independent of the fluctuation of
the total size $T$.  The complementary Fourier arcs are negligible by the
minor-arc estimate in Lemma~\ref{lem:Fourier-estimates}.

For $t=t_n$, the phase $\e^{-in\varphi}$ cancels the linear term in the
Taylor expansion of $J_t$ because $n=\mu(t_n)$.  The remaining quadratic
term is Gaussian on the natural central scale
$\varphi\asymp\sigma(t_n)^{-1}\asymp t_n^{3/2}$.  Away from that scale we
need uniform decay.  The following elementary estimate supplies it.

\subsection{A trigonometric lower bound}

\begin{lemma}\label{lem:trig}
There is an absolute constant $c>0$ such that, for every integer $K\geq1$
and every $|\varphi|\leq\pi$,
\begin{equation}\label{eq:trig-sum}
 \sum_{k=1}^K(1-\cos(k\varphi))
 \geq c\min\{K^3\varphi^2,K\}.
\end{equation}
\end{lemma}

\begin{proof}
We may assume $\varphi\geq0$ because the sum is even in $\varphi$.  We use
the elementary inequality
\begin{equation}\label{eq:cos-quadratic}
 1-\cos x\geq c_0x^2,\qquad |x|\leq1,
\end{equation}
for an absolute constant $c_0>0$.

First suppose $0\leq\varphi\leq K^{-1}$.  Then
$|k\varphi|\leq1$ for all $1\leq k\leq K$, and
\eqref{eq:cos-quadratic} gives
\[
 \sum_{k=1}^K(1-\cos(k\varphi))
 \geq c_0\varphi^2\sum_{k=1}^Kk^2
 \geq c_1K^3\varphi^2.
\]

Next suppose $K^{-1}\leq\varphi\leq4K^{-1}$.  For
$1\leq k\leq\lfloor K/4\rfloor$ we still have $k\varphi\leq1$.  Hence,
for $K\geq8$,
\[
 \sum_{k=1}^K(1-\cos(k\varphi))
 \geq c_0\varphi^2
 \sum_{k\leq K/4}k^2
 \geq c_2K^3\varphi^2
 \geq c_2K.
\]
For the finitely many cases $1\leq K<8$, the same conclusion follows by
decreasing the absolute constant $c$ to the minimum of the corresponding
positive continuous ratios on $[0,\pi]$.

Finally, suppose $4K^{-1}\leq\varphi\leq\pi$.  Since
\[
 \sum_{k=1}^K(1-\cos(k\varphi))
 =K-\operatorname{Re}\sum_{k=1}^K\e^{ik\varphi},
\]
the geometric-sum formula and
$\sin(x)\geq2x/\pi$ for $0\leq x\leq\pi/2$ give
\[
 \left|\sum_{k=1}^K\e^{ik\varphi}\right|
 \leq\frac{1}{|\sin(\varphi/2)|}
 \leq\frac{\pi}{|\varphi|}
 \leq\frac{\pi K}{4}.
\]
Therefore the sum is at least $(1-\pi/4)K$.  Combining the three ranges and
decreasing the absolute constant if necessary proves
\eqref{eq:trig-sum}.
\end{proof}

\subsection{Central and complementary Fourier arcs}

Let $L_t=\log(1/t)$.  The central arc
$|\varphi|\leq L_tt^{3/2}$ is wider than the Gaussian scale
$t^{3/2}$ by a slowly growing factor, but still narrow enough for a
third-order Taylor expansion to be uniform.  The complementary arc is the
rest of $[-\pi,\pi]$.  The next lemma treats both regions.  Its first part
says, in characteristic-function language, that $tD$ and
$t^{3/2}(T-\mu(t))$ become asymptotically independent: their joint limiting
characteristic function factors into the product of the two individual
characteristic functions.  Its second part is the uniform decay needed to
discard all noncentral frequencies.

The central arc is the neighborhood of $0$ in the Fourier variable.  Near
$\varphi=0$, the characteristic function has a quadratic approximation;
outside this neighborhood, the product estimate below gives the required
decay.

\begin{lemma}\label{lem:Fourier-estimates}
Uniformly for $\theta$ in a fixed compact interval,
as $t\downarrow0$ the following statements hold.

\begin{enumerate}
\item For $|u|\leq L_t$,
\begin{equation}\label{eq:central-expansion}
 \frac{\e^{-iu t^{3/2}\mu(t)}
 J_t(\theta,u t^{3/2})}{\Phi_t(\theta)}
 =\exp\left(-\frac{\pi^2u^2}{6}+o(1)\right),
\end{equation}
where the $o(1)$ is uniform in $u$ and $\theta$.

\item There are constants $c,C>0$ such that, whenever
$L_tt^{3/2}\leq|\varphi|\leq\pi$,
\begin{equation}\label{eq:minor-arc-bound}
 |J_t(\theta,\varphi)|
 \leq C\exp\left[-c\min\left(
 \frac{\varphi^2}{t^3},\frac1t\right)\right].
\end{equation}
\end{enumerate}
\end{lemma}

\begin{proof}
Fix $M>0$ and suppose $|\theta|\leq M$.  Write $\rho_k=\e^{-tk}$,
$\alpha_k=\theta t\varepsilon_k$.  Since
$\theta tD+\varphi T=\sum_{k\geq1}
(\alpha_k+k\varphi)X_k$, formula
\eqref{eq:single-geometric-pgf} gives the exact product
\begin{equation}\label{eq:joint-product}
 J_t(\theta,\varphi)
 =\prod_{k\geq1}
 \frac{1-\rho_k}
 {1-\rho_k\e^{i(k\varphi+\alpha_k)}}.
\end{equation}
The product converges locally uniformly in real $(\theta,\varphi)$ because
the logarithm of its $k$th factor is $O(\rho_k)$ for large $k$.

For the central expansion, put
\[
 h_k(z)=-\log(1-\rho_k\e^{iz}).
\]
Because $|\rho_k|<1$, the curve $1-\rho_k\e^{iz}$ stays in the open right
half-plane for real $z$; hence the principal branch of the logarithm is
well defined and smooth.  For large $k$ the corresponding logarithms are
$O(\rho_k)$, uniformly when $|\theta|\leq M$ and $|\varphi|\leq\pi$.
Thus the following series converges normally.  Define its sum by
\begin{equation}\label{eq:log-joint-ratio-def}
 \Xi_t(\theta,\varphi)
 :=\sum_{k\geq1}\{h_k(\alpha_k+k\varphi)-h_k(\alpha_k)\}.
\end{equation}
Exponentiating term by term gives
\begin{equation}\label{eq:joint-ratio-exponential}
 \frac{J_t(\theta,\varphi)}{\Phi_t(\theta)}
 =\exp\bigl(\Xi_t(\theta,\varphi)\bigr).
\end{equation}
Direct differentiation gives
\[
 h_k'(z)=\frac{i\rho_k\e^{iz}}{1-\rho_k\e^{iz}},
 \qquad
 h_k''(z)=-\frac{\rho_k\e^{iz}}
 {(1-\rho_k\e^{iz})^2},
\]
and
\[
 h_k'''(z)
 =-\frac{i\rho_k\e^{iz}(1+\rho_k\e^{iz})}
 {(1-\rho_k\e^{iz})^3}.
\]
In particular, at zero,
\[
 h_k'(0)=i\frac{\rho_k}{1-\rho_k},\qquad
 h_k''(0)=-\frac{\rho_k}{(1-\rho_k)^2}.
\]
For real $z$,
\begin{equation}\label{eq:h-derivative-bounds}
 |h_k''(z)|\leq C\frac{\rho_k}{(1-\rho_k)^2},
 \qquad
 |h_k'''(z)|\leq C\frac{\rho_k(1+\rho_k)}{(1-\rho_k)^3},
\end{equation}
because $|1-\rho_k\e^{iz}|\geq1-\rho_k$.

We now quantify the effect of the small shifts $\alpha_k$.  By the
mean-value theorem, \eqref{eq:h-derivative-bounds}, and
$|\alpha_k|\leq Mt$,
\[
 |h_k'(\alpha_k)-h_k'(0)|
 \leq Mt\sup_{z\in\mathbb R}|h_k''(z)|,
 \qquad
 |h_k''(\alpha_k)-h_k''(0)|
 \leq Mt\sup_{z\in\mathbb R}|h_k'''(z)|.
\]
Split the resulting sums at $k=1/t$.  When $k\leq1/t$, we have
$1-\rho_k=1-\e^{-tk}\asymp tk$; when $k>1/t$, the denominators in
\eqref{eq:h-derivative-bounds} are bounded away from zero and the numerators
decay like $\e^{-tk}$.  These observations give
\begin{align}
 \sum_{k\geq1}k|h_k'(\alpha_k)-h_k'(0)|
 &=O\left(t^{-1}\log\frac1t\right),
 \label{eq:first-mixed-bound}\\
 \sum_{k\geq1}k^2|h_k''(\alpha_k)-h_k''(0)|
 &=O\left(t^{-2}\log\frac1t\right),
 \label{eq:second-mixed-bound}\\
 \sum_{k\geq1}k^3\sup_{z\in\mathbb R}|h_k'''(z)|
 &=O(t^{-4}).
 \label{eq:third-bound}
\end{align}
For completeness, below the split the three summands are bounded,
respectively, by
\[
 \frac{C}{tk},\qquad \frac{C}{t^2k},\qquad \frac{C}{t^3}.
\]
Their sums are $O(t^{-1}\log(1/t))$,
$O(t^{-2}\log(1/t))$, and $O(t^{-4})$.  Above the split, comparison with
$\sum_{k>1/t} k^j\e^{-tk}=O(t^{-j-1})$ gives bounds no larger than these;
indeed, an integral comparison gives
\[
 \sum_{k>1/t}k^j\e^{-tk}
 \leq C_j\int_{1/t}^{\infty}x^j\e^{-tx}\,dx
 =O(t^{-j-1}).
\]
Here and below all constants may depend on the fixed bound $M$, but not on
$t$, $u$, or $\theta$ with $|\theta|\leq M$.

Taylor's formula at $\alpha_k$ reads
\[
 \begin{aligned}
 h_k(\alpha_k+k\varphi)-h_k(\alpha_k)
 ={}&k\varphi h_k'(\alpha_k)
 +\frac{k^2\varphi^2}{2}h_k''(\alpha_k)+R_k,\\
 |R_k|\leq{}&C|k\varphi|^3
 \sup_{z\in\mathbb R}|h_k'''(z)|.
 \end{aligned}
\]
The identities
\[
 \sum_{k\geq1}kh_k'(0)=i\mu(t),\qquad
 \sum_{k\geq1}k^2h_k''(0)=-\sigma^2(t)
\]
identify the linear and quadratic main terms.  Summing the Taylor formula
and applying \eqref{eq:first-mixed-bound}--\eqref{eq:third-bound} gives
\begin{align}
 \Xi_t(\theta,\varphi)-i\varphi\mu(t)
 ={}&-\frac12\sigma^2(t)\varphi^2
 +O\left(|\varphi|t^{-1}\log\frac1t\right)\notag\\
 &+O\left(\varphi^2t^{-2}\log\frac1t\right)
 +O(|\varphi|^3t^{-4}).
 \label{eq:Taylor-central}
\end{align}
The first term is the contribution of the variance of $T$; the remaining
terms are error terms.  Since $\sigma^2(t)\asymp t^{-3}$, the natural scale
for $\varphi$ is $t^{3/2}$.
Take $\varphi=ut^{3/2}$ and $|u|\leq L_t$.  Every error term in
\eqref{eq:Taylor-central} tends to zero uniformly, because the three error
bounds become, respectively,
\[
 O\left(t^{1/2}L_t\log\frac1t\right),\qquad
 O\left(tL_t^2\log\frac1t\right),\qquad
 O(t^{1/2}L_t^3).
\]
In addition,
Lemma~\ref{lem:moments} gives
\[
 \sup_{|u|\leq L_t}
 \left|\left(t^3\sigma^2(t)-\frac{\pi^2}{3}\right)u^2\right|
 =O(tL_t^2)=o(1).
\]
This proves \eqref{eq:central-expansion}.

The same estimate supplies the domination needed later.  Taking real parts
in \eqref{eq:Taylor-central}, using
$t^3\sigma^2(t)\to\pi^2/3$ from Lemma~\ref{lem:moments}, and using the
uniform error bounds just listed, gives, for all sufficiently small $t$ and
$|u|\leq L_t$,
\[
 \Re\left(\Xi_t(\theta,u t^{3/2})
       -iu t^{3/2}\mu(t)\right)
 \leq -c u^2+C
\]
with constants independent of $u$ and $\theta$ in the fixed compact set.
Indeed, the three error terms are uniformly $o(1)$ on this range, while
the quadratic term is $-(\pi^2/6+o(1))u^2$; after reducing $c>0$ and
enlarging the constant to cover bounded $u$, the displayed inequality
follows.
Consequently the rescaled central integrand is bounded by
$C'\e^{-cu^2}$.

It remains to prove \eqref{eq:minor-arc-bound}.  For a real phase $x$,
\[
 |1-\rho_k\e^{ix}|^2
 =(1-\rho_k)^2+2\rho_k(1-\cos x).
\]
Taking absolute values factor by factor in \eqref{eq:joint-product}
therefore gives
\begin{equation}\label{eq:modulus-product}
 |J_t(\theta,\varphi)|^2
 =\prod_{k\geq1}
 \frac{(1-\rho_k)^2}
 {(1-\rho_k)^2+2\rho_k\{1-\cos(k\varphi+\alpha_k)\}}.
\end{equation}
Every factor in this product is at most one, so discarding factors can only
increase the right-hand side.  Restrict to
$1\leq k\leq K=\lfloor1/t\rfloor$.  On this range
$\rho_k\geq\e^{-1}$ and $(1-\rho_k)^2\leq1$, whence
\[
 \frac{2\rho_k(1-\cos x)}{(1-\rho_k)^2}
 \geq2\e^{-1}(1-\cos x).
\]
Writing each retained factor as $(1+A_k)^{-1}$, taking logarithms, and
using $\log(1+cx)\geq c'x$ for $0\leq x\leq2$, with
$x=1-\cos(k\varphi+\alpha_k)$, we obtain (after taking
the square root of the product estimate)
\begin{equation}\label{eq:product-to-trig}
 |J_t(\theta,\varphi)|
 \leq\exp\left[-c_2\sum_{k=1}^K
 \{1-\cos(k\varphi+\alpha_k)\}\right].
\end{equation}
Thus the probabilistic estimate has been reduced to a deterministic
trigonometric sum.  The previous lemma says that many of the phases cannot
all be close to multiples of $2\pi$ unless $\varphi$ is very small.  This
is exactly the cancellation mechanism that makes the noncentral Fourier
frequencies negligible.
The shifts $\alpha_k$ do not destroy the trigonometric lower bound.  Indeed,
since $|\cos(x+y)-\cos x|\leq|y|$ and
$|\alpha_k|\leq Mt$,
Lemma~\ref{lem:trig} implies
\begin{align*}
 \sum_{k=1}^K\{1-\cos(k\varphi+\alpha_k)\}
 &\geq c_3\min\{K^3\varphi^2,K\}-MKt\\
 &\geq c_4\min\left\{\frac{\varphi^2}{t^3},\frac1t\right\}-C_M.
\end{align*}
Here $K\asymp t^{-1}$ as $t\downarrow0$, which justifies the second line.
On the stated complementary arc,
$|\varphi|\geq L_tt^{3/2}$, so the minimum is at least a constant times
$L_t^2$.  It diverges, whereas the error $C_M$ is bounded.  After decreasing
$c_4$ and enlarging a fixed prefactor $C$, the error is absorbed.
Combining the result with \eqref{eq:product-to-trig} proves
\eqref{eq:minor-arc-bound}.
\end{proof}

\subsection{Fourier inversion and deconditioning}

We can now return to the ratio in \eqref{eq:Fourier-roadmap}.  The following
proposition is the precise deconditioning statement: after the saddle-point
choice $t=t_n$, imposing the exact value $T=n$ does not alter the limiting
characteristic function of $tD$.

\begin{proposition}
\label{prop:conditioning}
For every fixed $\theta\in\mathbb R$,
\begin{equation}\label{eq:conditional-characteristic}
 \EE_{t_n}\left(\e^{i\theta t_nD}\mid T=n\right)
 \longrightarrow\phi(\theta),
\end{equation}
where $\phi$ is given by \eqref{eq:phi-limit}.
\end{proposition}

\begin{proof}
Let
\[
 \mathfrak C_t=\{\varphi:|\varphi|\leq L_tt^{3/2}\},
 \qquad
 \mathfrak c_t=[-\pi,\pi]\setminus\mathfrak C_t
\]
denote the central and complementary arcs.  We first apply the second line
of \eqref{eq:Fourier-roadmap}.  Take $\theta=0$; then
$\Phi_t(0)=1$.  Since $\mu(t_n)=n$, the oscillating factor
$\e^{-in\varphi}$ exactly produces the centered expression on the left of
\eqref{eq:central-expansion}.  On $\mathfrak C_{t_n}$, substitute
$\varphi=u t_n^{3/2}$, so that $d\varphi=t_n^{3/2}\,du$ and
$|u|\leq L_{t_n}$.  The first assertion of
Lemma~\ref{lem:Fourier-estimates}, namely the central-arc expansion
\eqref{eq:central-expansion}, gives the uniform estimate
\[
 \e^{-in u t_n^{3/2}}J_{t_n}(0,u t_n^{3/2})
 =\exp\left(-\frac{\pi^2u^2}{6}+o(1)\right),
 \qquad |u|\leq L_{t_n},
\]
because $\mu(t_n)=n$ and $\Phi_{t_n}(0)=1$.  The second assertion of
the same lemma, \eqref{eq:minor-arc-bound}, shows that the integral over
$\mathfrak c_{t_n}$ is $o(t_n^{3/2})$.  Therefore the central integral is
\[
 \frac{t_n^{3/2}}{2\pi}
 \int_{-L_{t_n}}^{L_{t_n}}
 \exp\left(-\frac{\pi^2u^2}{6}+o(1)\right)du,
\]
and the Gaussian domination established in Lemma~\ref{lem:Fourier-estimates}
allows us to extend the limits to $\pm\infty$ and apply dominated
convergence.  Thus
\begin{align}
 \PP_{t_n}(T=n)
 &=\frac1{2\pi}\int_{-\pi}^{\pi}
   \e^{-in\varphi}J_{t_n}(0,\varphi)\,d\varphi\\
 &\sim\frac{t_n^{3/2}}{2\pi}
 \int_{-\infty}^{\infty}\exp\left(-\frac{\pi^2u^2}{6}\right)du\\
 &=\frac{t_n^{3/2}}
 {\sqrt{2\pi(\pi^2/3)}}.
 \label{eq:local-limit-proof}
\end{align}
For the last equality, use the elementary Gaussian identity
\[
 \int_{-\infty}^{\infty}e^{-a u^2}\,du=\sqrt{\frac{\pi}{a}}
 \quad(a>0),
 \qquad
 \int_{-\infty}^{\infty}e^{-\pi^2u^2/6}\,du
 =\sqrt{\frac{6}{\pi}}.
\]
Consequently,
\[
 \frac1{2\pi}\sqrt{\frac6\pi}
 =\frac1{\sqrt{2\pi(\pi^2/3)}},
\]
which is the constant in \eqref{eq:local-limit-proof}.

The Fourier peak has width $t^{3/2}$, and hence its integral, which gives the
lattice point probability $\PP_t(T=n)$, is of order $t^{3/2}$.  A global
central limit theorem does not by itself give this single-point estimate.

For completeness, the domination just used follows from the real part of
the Taylor estimate in the proof of
Lemma~\ref{lem:Fourier-estimates}; for all sufficiently small $t$, the
rescaled central integrand is bounded by $C\e^{-cu^2}$.
The integral over $\mathfrak c_{t_n}$ is negligible.  Indeed,
\eqref{eq:minor-arc-bound}, split at $|\varphi|=t_n$ and followed by
$u=\varphi/t_n^{3/2}$ on the first portion, gives the bound
\[
 Ct_n^{3/2}\int_{L_{t_n}}^\infty\e^{-cu^2}\,du
 +C\e^{-c/t_n}=o(t_n^{3/2}).
\]
This proves the displayed local limit formula for the denominator.

Now fix $\theta\in\mathbb R$ and apply the first line of
\eqref{eq:Fourier-roadmap}.  On the central arc,
\eqref{eq:central-expansion} has the same Gaussian factor, multiplied by
$\Phi_{t_n}(\theta)$.  The complementary-arc estimate is uniform for
$\theta$ in a compact set and is therefore still $o(t_n^{3/2})$.  It follows
that
\begin{align*}
 \EE_{t_n}\left(\e^{i\theta t_nD};T=n\right)
 &=\frac1{2\pi}\int_{-\pi}^{\pi}
 \e^{-in\varphi}J_{t_n}(\theta,\varphi)\,d\varphi\\
 &\sim\Phi_{t_n}(\theta)
 \frac{t_n^{3/2}}{\sqrt{2\pi(\pi^2/3)}}.
\end{align*}
Dividing this numerator estimate by the positive denominator
\eqref{eq:local-limit-proof} gives
\[
 \EE_{t_n}(\e^{i\theta t_nD}\mid T=n)
 =\Phi_{t_n}(\theta)(1+o(1)).
\]
There is no division-by-zero issue here: the $q$-Pochhammer representation
shows that $\Phi_t(\theta)\neq0$ for real $\theta$, and the locally uniform
limit in Proposition~\ref{prop:unconditioned} is nonzero.  Finally, using
$\Phi_{t_n}(\theta)\to\phi(\theta)$ from
Proposition~\ref{prop:unconditioned} proves
\eqref{eq:conditional-characteristic}.
\end{proof}

\begin{remark}
The estimate in the proof also gives the local limit formula
\[
 \PP_{t_n}(T=n)\sim
 \frac{t_n^{3/2}}{\sqrt{2\pi(\pi^2/3)}}.
\]
Combined with the exact identity
\[
 p(n)=Z(t_n)\e^{nt_n}\PP_{t_n}(T=n),
\]
this is the probabilistic saddle-point mechanism behind the Hardy--Ramanujan
asymptotic formula
\cite{HardyRamanujan1918}.
\end{remark}

\section{Proof of the main theorem and the commutation conjecture}
\label{sec:completion}

\begin{proof}[Proof of Theorem~\ref{thm:main}]
By the exact conditioning identity \eqref{eq:conditioning-identity}, the
conditional distribution of $D$ given $T=n$ under $\PP_{t_n}$ is the
distribution of $D(\Lambda_n)$.  In particular,
\[
 \EE\e^{i\theta t_nD(\Lambda_n)}
 =\EE_{t_n}(\e^{i\theta t_nD}\mid T=n).
\]
Proposition~\ref{prop:conditioning} shows that the right-hand side converges
to $\phi(\theta)$ for every real $\theta$.  Since $\phi$ is the
characteristic function of $W$, Levy's continuity theorem gives
\[
 t_nD(\Lambda_n)\xrightarrow{\ d\ }W.
\]
The variable $W=m^{-1}(\log G_b-\log G_a)$ has the continuous density
\eqref{eq:W-density}; in particular, $\PP(W=0)=0$, so the distribution has
no point mass (or atom) at zero.  The interval $(0,\infty)$ is therefore a
continuity set: its boundary $\{0\}$ has probability zero under the limiting
law.  Since
$t_n>0$, the events $\{D(\Lambda_n)>0\}$ and
$\{t_nD(\Lambda_n)>0\}$ coincide, and the portmanteau theorem yields
\[
 \lim_{n\to\infty}\PP(D(\Lambda_n)>0)=\PP(W>0).
\]
Here we invoke the following consequence of the portmanteau theorem (see
\cite[Chapter 2]{Billingsley1999}): if
$Y_n\xrightarrow{\ d\ }Y$ and $A$ is a Borel set whose boundary satisfies
$\PP(Y\in\partial A)=0$, then $\PP(Y_n\in A)\to\PP(Y\in A)$.  Applied
with $A=(0,\infty)$, this explains why the absence of an atom at the
decision boundary is exactly what is needed here.
It remains only to evaluate the last probability.  By monotonicity of the
logarithm,
\[
 W>0
 \quad\Longleftrightarrow\quad
 G_a<G_b
 \quad\Longleftrightarrow\quad
 \frac{G_a}{G_a+G_b}<\frac12.
\]
The beta--gamma identity proved in
Proposition~\ref{prop:finite-limit} shows that the ratio in the last line has
the $\operatorname{Beta}(a,b)$ distribution.  Its probability of lying in
$(0,1/2)$ is, by definition, $I_{1/2}(a,b)$.  This proves
\eqref{eq:bias-main}.  Formula \eqref{eq:tn-asymp-intro} was proved in
Lemma~\ref{lem:moments}, so all assertions of the theorem follow.
\end{proof}

We now combine the unrestricted calculation with the independently obtained
finite-truncation limit.  This is the precise point at which the two halves
of the paper meet.

\begin{corollary}
\label{cor:commutation}
With $C_{n,N}$ as in Section~\ref{sec:truncation},
\begin{equation}\label{eq:commutation}
 \lim_{n\to\infty}\lim_{N\to\infty}C_{n,N}
 =\lim_{N\to\infty}\lim_{n\to\infty}C_{n,N}
 =I_{1/2}(r/m,s/m).
\end{equation}
In particular, Conjecture 1 of Li and Zhao \cite{LiZhao2026} holds.
\end{corollary}

\begin{proof}
For each fixed $n$, equation \eqref{eq:truncate-first} gives
\[
 \lim_{N\to\infty}C_{n,N}
 =\frac{p_{\cR>\cS}(n)}{p(n)}.
\]
Taking $n\to\infty$ and applying Theorem~\ref{thm:main} evaluates the first
iterated limit in \eqref{eq:commutation} as $I_{1/2}(a,b)$.  In the reverse
order, the inner limit is $L_N$ from \eqref{eq:Li-Zhao-formula}, and
Proposition~\ref{prop:finite-limit} shows that
$L_N\to I_{1/2}(a,b)$.  Both orders therefore give the same value.
\end{proof}

\section{Worked examples and exact computations}
\label{sec:numerics}

The beta law is explicit enough to produce concrete constants, and the
partition ratios themselves can be computed exactly for moderately large
$n$.  We give both kinds of examples.  Besides illustrating the size of the
bias, the exact coefficient experiment reveals finite-size oscillations
that are invisible in the limiting statement.

\subsection{Three explicit beta constants}

Write
\[
 B_x(a,b)=\int_0^x t^{a-1}(1-t)^{b-1}\,dt,
 \qquad I_x(a,b)=\frac{B_x(a,b)}{B(a,b)}.
\]
When $a+b=1$, the substitution $y=t/(1-t)$ gives the useful identity
\begin{equation}\label{eq:beta-half-transform}
 B_{1/2}(a,1-a)=\int_0^1\frac{y^{a-1}}{1+y}\,dy.
\end{equation}

\begin{example}
\label{ex:m3}
Take $(m,r,s)=(3,1,2)$.  Then $(a,b)=(1/3,2/3)$.  Euler's reflection
formula
\[
 \Gamma(z)\Gamma(1-z)=\frac{\pi}{\sin(\pi z)}
\]
(see \cite[Section 5.5]{Olver2010}) gives
\[
 B\left(\frac13,\frac23\right)
 =\Gamma\left(\frac13\right)\Gamma\left(\frac23\right)
 =\frac{2\pi}{\sqrt3}.
\]
In \eqref{eq:beta-half-transform}, put $y=u^3$.  A partial-fraction
integration gives
\[
 \begin{aligned}
 B_{1/2}\left(\frac13,\frac23\right)
 &=3\int_0^1\frac{du}{1+u^3}\\
 &=3\left(\frac{\log2}{3}+\frac{\pi}{3\sqrt3}\right)
 =\log2+\frac{\pi}{\sqrt3}.
 \end{aligned}
\]
Consequently,
\begin{equation}\label{eq:m3-constant}
 I_{1/2}\left(\frac13,\frac23\right)
 =\frac12+\frac{\sqrt3\log2}{2\pi}
 =0.6910760347\ldots.
\end{equation}
Thus $I_{1/2}(1/3,2/3)\approx 0.691$.
\end{example}

\begin{example}
\label{ex:m4-13}
For $(m,r,s)=(4,1,3)$, we have $(a,b)=(1/4,3/4)$ and
\[
 B\left(\frac14,\frac34\right)
 =\frac{\pi}{\sin(\pi/4)}=\pi\sqrt2.
\]
Putting $y=u^4$ in \eqref{eq:beta-half-transform} gives
\[
 B_{1/2}\left(\frac14,\frac34\right)
 =4\int_0^1\frac{du}{1+u^4}.
\]
Factoring $1+u^4$ over the reals, or using partial fractions, yields
\[
 \int_0^1\frac{du}{1+u^4}
 =\frac{\pi}{4\sqrt2}
  +\frac{\log(1+\sqrt2)}{2\sqrt2}.
\]
It follows that
\begin{equation}\label{eq:m4-13-constant}
 I_{1/2}\left(\frac14,\frac34\right)
 =\frac12+\frac{\log(1+\sqrt2)}{\pi}
 =0.7805499262\ldots.
\end{equation}
The limiting bias is now about $78.1\%$, appreciably stronger than in
Example~\ref{ex:m3}.
\end{example}

\begin{example}
\label{ex:m4-12}
For $(m,r,s)=(4,1,2)$, the parameters are $(a,b)=(1/4,1/2)$.  The
denominator is
\[
 B\left(\frac14,\frac12\right)
 =\frac{\sqrt\pi\,\Gamma(1/4)}{\Gamma(3/4)}.
\]
In the numerator, the substitution $t=u^4$ removes the algebraic
singularity at zero:
\[
 B_{1/2}\left(\frac14,\frac12\right)
 =4\int_0^{2^{-1/4}}\frac{du}{\sqrt{1-u^4}}.
\]
This is an elliptic integral rather than an elementary one.  Numerically,
\begin{equation}\label{eq:m4-12-constant}
 I_{1/2}\left(\frac14,\frac12\right)
 =\frac{4\Gamma(3/4)}
 {\sqrt\pi\,\Gamma(1/4)}
 \int_0^{2^{-1/4}}\frac{du}{\sqrt{1-u^4}}
 =0.6831128993\ldots.
\end{equation}
\end{example}

\subsection{An exact coefficient experiment}

We next compare the limiting constants with finite partition counts.  This
is an exact enumeration, not a Monte Carlo simulation.  Define
\[
 \varepsilon_k=\1_{\{k\in\cR\}}-\1_{\{k\in\cS\}}
\]
as in \eqref{eq:epsilon}, and introduce the bivariate generating function
\begin{equation}\label{eq:bivariate-experiment}
 \mathcal F(u,z)
 =\prod_{k\geq1}\frac1{1-z^{\varepsilon_k}u^k}
 =\sum_{n\geq0}\sum_{d\in\mathbb Z}c_n(d)u^nz^d.
\end{equation}
A choice of the term
$z^{j\varepsilon_k}u^{jk}$ from the $k$th factor means that the part $k$
occurs $j$ times.  Therefore
\[
 c_n(d)=|\{\lambda\in\mathcal P(n):D(\lambda)=d\}|.
\]

For reproducibility, the coefficients were computed by the following
unbounded-part recurrence.  Let $c_K(n,d)$ count partitions using only
parts at most $K$.  With $c_0(0,0)=1$ and all other initial values zero,
\begin{equation}\label{eq:exact-DP}
 c_K(n,d)=c_{K-1}(n,d)
 +\1_{\{n\geq K\}}c_K(n-K,d-\varepsilon_K).
\end{equation}
The first term counts partitions with no part $K$; after removing one copy
of $K$, the second term still permits further copies of $K$.  For
$0\leq n\leq N$, iterating \eqref{eq:exact-DP} through $K=N$ therefore
computes $c_n(d)$ exactly.  As an independent check, at every computed
value of $n$ we verified
\begin{equation}\label{eq:partition-checksum}
 \sum_d c_n(d)=p(n),
\end{equation}
where $p(n)$ was separately generated by Euler's pentagonal-number
recurrence (see \cite{Andrews1976}).  All calculations used arbitrary-precision integers; decimal
rounding enters only when the final ratios are displayed.  The code used to
produce the tables, \texttt{partition\_bias\_experiment.py}, is included
with the submission as supplementary material and implements
\eqref{eq:exact-DP}.

\subsection{The modulo \texorpdfstring{$3$}{3} experiment}

Table~\ref{tab:m3-experiment} reports
\[
 P_n^+=\frac{\sum_{d>0}c_n(d)}{p(n)},\qquad
 P_n^0=\frac{c_n(0)}{p(n)}
\]
for $(m,r,s)=(3,1,2)$.  A useful finite-size feature can be predicted
without computation.  Reducing the total size modulo $3$ gives
\[
 n\equiv N_{\cR}(\lambda)+2N_{\cS}(\lambda)
 \equiv D(\lambda)\pmod3.
\]
Thus a tie is possible only when $3\mid n$.  This congruence produces three
visible subsequences in the displayed data: the classes $n\equiv1,2\pmod3$
approach the limit from opposite sides, while the class $n\equiv0\pmod3$
also carries a positive tie probability.  Theorem~\ref{thm:main}, together
with Corollary~\ref{cor:direction}, implies that the common limiting bias is
the beta value and that the tie probability vanishes asymptotically.

\begin{table}[ht]
\centering
\small
\begin{tabular}{rcrrr}
\hline
$n$ & $n\bmod3$ & $P_n^+$ & $P_n^0$ &
$P_n^+-I_{1/2}(1/3,2/3)$ \\
\hline
198  & 0 & 0.631412 & 0.110941 & $-0.059664$ \\
199  & 1 & 0.706992 & 0        & $+0.015916$ \\
200  & 2 & 0.669956 & 0        & $-0.021120$ \\
398  & 2 & 0.675964 & 0        & $-0.015112$ \\
399  & 0 & 0.649016 & 0.078617 & $-0.042060$ \\
400  & 1 & 0.702244 & 0        & $+0.011168$ \\
798  & 0 & 0.661330 & 0.055813 & $-0.029746$ \\
799  & 1 & 0.698951 & 0        & $+0.007875$ \\
800  & 2 & 0.680340 & 0        & $-0.010736$ \\
1598 & 2 & 0.683440 & 0        & $-0.007636$ \\
1599 & 0 & 0.670063 & 0.039537 & $-0.021013$ \\
1600 & 1 & 0.696629 & 0        & $+0.005553$ \\
\hline
\end{tabular}
\caption{Exact partition ratios for $(m,r,s)=(3,1,2)$.  The predicted
limit is $0.6910760347\ldots$ from \eqref{eq:m3-constant}.}
\label{tab:m3-experiment}
\end{table}

The trend is consistent with both conclusions of
Theorem~\ref{thm:main}.  At the largest displayed values, the deviations of
the two nonzero residue classes have fallen to about $0.006$ and $0.008$,
while along the divisible-by-$3$ subsequence the tie probability decreases
from $0.111$ at $n=198$ to $0.0395$ at $n=1599$.  The latter behavior also
supports Corollary~\ref{cor:direction}'s assertion that ties have asymptotic
density zero.

\subsection{Comparison of the three examples}

For context, Table~\ref{tab:three-experiments} compares the three constants
above with exact finite-$n$ ratios.  The convergence need not be monotone:
the marked statistic is lattice-valued, and the relevant congruence class of
$n$ may change between rows.  What is important is the shrinking distance
to the corresponding beta value.

\begin{table}[ht]
\centering
\small
\begin{tabular}{ccrrr}
\hline
$(m,r,s)$ & beta-law limit & $P_{200}^+$ & $P_{400}^+$ &
$P_{800}^+$ \\
\hline
$(3,1,2)$ & 0.691076 & 0.669956 & 0.702244 & 0.680340 \\
$(4,1,2)$ & 0.683113 & 0.660300 & 0.666911 & 0.671622 \\
$(4,1,3)$ & 0.780550 & 0.733590 & 0.747546 & 0.757315 \\
\hline
\end{tabular}
\caption{Exact finite-size ratios compared with the limiting beta
probabilities.}
\label{tab:three-experiments}
\end{table}

These computations are not used in the proof.  They serve instead as
independent checks of the normalization, the direction and magnitude of the
bias, and the disappearance of the equality case.

\section{Consequences and further remarks}
\label{sec:further}

The beta expression contains more information than the commutation identity
alone.  We close by drawing out several consequences and by relating the
Boltzmann picture to the finite-set formula.

\subsection{The direction of the bias and the probability of a tie}

\begin{corollary}\label{cor:direction}
Under the assumptions of Theorem~\ref{thm:main},
\[
 I_{1/2}(r/m,s/m)
 \begin{cases}
 >1/2,&r<s,\\
 <1/2,&r>s.
 \end{cases}
\]
Moreover, the proportion of partitions for which the two multiplicity
counts are equal tends to zero.
\end{corollary}

\begin{proof}
Suppose first that $r<s$, or equivalently $a<b$.  Let $G_a'$ have the same
distribution as $G_a$, and let $H\sim\operatorname{Gamma}(b-a,1)$ be
independent of both.  Additivity of independent gamma variables with the
same rate means that their shapes add under summation: the density of their
sum is again a gamma density, with shape equal to the sum of the two shapes.
It gives
$G_a'+H\sim\operatorname{Gamma}(b,1)$, so we may realize
$G_b=G_a'+H$.  Since $H>0$ almost surely,
\[
 \PP(G_a<G_b)>\PP(G_a<G_a')=\frac12;
\]
the last equality follows because $G_a$ and $G_a'$ are independent,
identically distributed, and continuous.  Theorem~\ref{thm:main} identifies
the left-hand side with $I_{1/2}(a,b)$.  If $r>s$, interchange $a$ and $b$
and use
\[
 I_{1/2}(a,b)+I_{1/2}(b,a)=1,
\]
which follows either from the beta density under $x\mapsto1-x$ or from
$\PP(G_a<G_b)+\PP(G_b<G_a)=1$.

For the tie probability, fix $\delta>0$.  The event
$\{D(\Lambda_n)=0\}$ is contained in
$\{|t_nD(\Lambda_n)|\leq\delta\}$.  The latter is a closed interval event,
so the portmanteau theorem gives
\[
 \limsup_{n\to\infty}\PP(D(\Lambda_n)=0)
 \leq\limsup_{n\to\infty}
 \PP(|t_nD(\Lambda_n)|\leq\delta)
 \leq\PP(|W|\leq\delta).
\]
Letting $\delta\downarrow0$ proves the claim because the continuous random
variable $W$ has no atom at zero.
\end{proof}

\subsection{The special case \texorpdfstring{$s=m$}{s=m}}

\begin{remark}
If $b=1$, then
\[
 I_{1/2}(a,1)=\int_0^{1/2}a u^{a-1}\,du=2^{-a}.
\]
Thus \eqref{eq:bias-main} reduces to $2^{-r/m}$, the value predicted in
\cite{LiZhao2026}.  When $r=1$, this is the constant $2^{-1/m}$ obtained in
\cite{KimKim2021}.
\end{remark}

\subsection{Fixed part sizes and the finite-set formula}

The conditioning method also gives a transparent interpretation of the
finite-set formula in Li and Zhao's Theorem 1.  We begin with a standard
finite-dimensional limit for unrestricted random partitions.

\begin{proposition}\label{prop:fixed-parts}
Let $k_1,\ldots,k_d$ be fixed distinct positive integers.  Under the uniform
measure on $\mathcal P(n)$,
\begin{equation}\label{eq:finite-parts-exp}
 (t_nN_{\{k_1\}}(\Lambda_n),\ldots,
  t_nN_{\{k_d\}}(\Lambda_n))
 \xrightarrow{\ d\ }(Y_{k_1},\ldots,Y_{k_d}),
\end{equation}
where the $Y_{k_j}$ are independent exponentials of respective rates
$k_j$.  An exponential variable of rate $k>0$ is a positive random variable
with density $k\e^{-ky}$ for $y>0$, or equivalently with tail probability
$\PP(Y_k>x)=\e^{-kx}$; its mean is $1/k$.
\end{proposition}

\begin{proof}
First work under the unconditioned measure $\PP_t$.  For $x\geq0$,
\[
 \begin{aligned}
 \PP_t(tX_{k_j}>x)
 &=\PP_t\left(X_{k_j}\geq\left\lfloor\frac{x}{t}\right\rfloor+1\right)\\
 &=\exp\left[-tk_j
 \left(\left\lfloor\frac{x}{t}\right\rfloor+1\right)\right]
 \longrightarrow\e^{-k_jx}.
 \end{aligned}
\]
Thus $tX_{k_j}$ converges to an exponential variable of rate $k_j$.
Because the $X_{k_j}$ are independent for every $t$, the vector converges
to independent exponentials.

It remains to condition on $T=n$.  One may insert the finite marking factor
\[
 \exp\left(i t\sum_{j=1}^d\theta_jX_{k_j}\right)
\]
in the joint characteristic function \eqref{eq:J-def}.  More explicitly,
for $(\theta_1,\ldots,\theta_d)$ in a fixed compact set, the phase at $k$
is changed by
\[
 \alpha_k=t\sum_{j=1}^d\theta_j\1_{\{k=k_j\}},
 \qquad |\alpha_k|\leq Mt,
\]
for a constant $M$ independent of $t$.  Only the finitely many phases
corresponding to $k_1,\ldots,k_d$ are changed.  The proof of
Lemma~\ref{lem:Fourier-estimates} uses the phase shifts only through this
bound and through $\sum_k|\alpha_k|=O(t)$; hence its central-arc Taylor
expansion and complementary-arc bound remain valid uniformly on the compact
set.  Repeating the numerator and denominator Fourier inversions in
Proposition~\ref{prop:conditioning} shows that conditioning on $T=n$ does
not change the limiting joint characteristic function.  Finally use
\eqref{eq:conditioning-identity} to identify the conditioned multiplicities
with those of $\Lambda_n$.  This proves \eqref{eq:finite-parts-exp}.
\end{proof}

If $R$ and $S$ are finite disjoint sets, summing the corresponding
coordinates in Proposition~\ref{prop:fixed-parts} and using the fact that
two independent sums of continuous random variables are equal with
probability zero gives
\begin{equation}\label{eq:finite-probability-interpretation}
 \lim_{n\to\infty}
 \PP\left(N_R(\Lambda_n)>N_S(\Lambda_n)\right)
 =\PP\left(\sum_{r\in R}Y_r>\sum_{s\in S}Y_s\right).
\end{equation}

Although the left-hand side of
\eqref{eq:finite-probability-interpretation} concerns unrestricted
partitions, the probability on the right also equals the volume ratio that
arises in the finite allowed-parts model.  To see why, let
$A=R\cup S\cup I$ be finite and set $Z_k=kY_k$ for $k\in A$.  Since
$Y_k\sim\operatorname{Exp}(k)$, the variables $Z_k$ are independent
$\operatorname{Exp}(1)$ variables.  Their normalized vector
\[
 U_k=\frac{Z_k}{\sum_{j\in A}Z_j},\qquad k\in A,
\]
is uniform on the simplex
$\{(u_k):u_k\geq0, \sum_{k\in A}u_k=1\}$; equivalently, it has the
$\operatorname{Dirichlet}(1,\ldots,1)$ distribution.  Thus the Dirichlet
name here introduces no additional weighting: equal-volume regions of the
simplex have equal probability.  Moreover, the event
in \eqref{eq:finite-probability-interpretation} is homogeneous and can be
written as
\[
 \sum_{r\in R}\frac{U_r}{r}>
 \sum_{s\in S}\frac{U_s}{s}.
\]
This is exactly the half-space condition obtained after scaling a restricted
partition multiplicity vector by
$u_k=kN_{\{k\}}/n$.  The diagonal change of variables has constant
Jacobian, which cancels from the numerator-to-denominator volume ratio.
Hence the probability involving these sums of independent exponential
variables in
\eqref{eq:finite-probability-interpretation} equals Li and Zhao's limiting
polytope-volume ratio.

One can also recover their explicit expression directly.  Write
$R=\{r_1,\ldots,r_\ell\}$ and put
$U=\sum_{i=1}^\ell Y_{r_i}$.  Partial fractions applied to
$\prod_i r_i/(z+r_i)$ give
\[
 \PP(U>x)=\sum_{i=1}^\ell A_i\e^{-r_ix},
 \qquad
 A_i=\prod_{j\ne i}\frac{r_j}{r_j-r_i}.
\]
If $V=\sum_{s\in S}Y_s$, independence and the Laplace transform
$\EE\e^{-zY_s}=s/(s+z)$ yield
\begin{equation}\label{eq:hypoexponential-formula}
 \begin{aligned}
 \PP(U>V)
 &=\EE\bigl[\PP(U>V\mid V)\bigr]\\
 &=\sum_{i=1}^\ell
 \left(\prod_{j\ne i}\frac{r_j}{r_j-r_i}\right)
 \left(\prod_{s\in S}\frac{s}{s+r_i}\right).
 \end{aligned}
\end{equation}
The Laplace transform of a nonnegative random variable $Y$ is the function
$z\mapsto\EE(\e^{-zY})$ for $z\geq0$; independence turns the transform of a
sum into the product of the transforms.  A sum of independent exponential
variables with possibly different rates is called hypoexponential.  Thus
\eqref{eq:hypoexponential-formula} is the hypoexponential form of the
finite-set formula; rearranging its products gives the expression in
\cite[Theorem 1]{LiZhao2026}.

This interpretation also makes their proposed bias criterion immediate.  If
$R=\{r_1,\ldots,r_d\}$ and $S=\{s_1,\ldots,s_d\}$ can be indexed so that
$r_j<s_j$ for every $j$, then the sum over $R$ strictly stochastically
dominates an independent copy $V'$ of the sum $V$ over $S$.  Here strict
stochastic dominance means that $\PP(U>x)\geq\PP(V'>x)$ for every $x$, with
strict inequality on some interval.  Indeed, an exponential variable of
rate $r_j$ strictly stochastically dominates one of rate $s_j$, and the
coordinatewise coupling preserves this order after summation.  If $V$ is
the independent sum appearing in
\eqref{eq:finite-probability-interpretation}, then
\[
 \PP(U>V)\geq\PP(V'>V)=\frac12,
\]
and the strict tail inequality makes the first inequality strict.

\subsection{Extensions}

The preceding arguments also suggest extensions, but we record them only as
remarks rather than as additional theorems.  If finitely many new fixed part
sizes are added to either progression, the finite-dimensional limit in
Proposition~\ref{prop:fixed-parts} suggests that the limiting variable is
augmented by a finite sum of independent exponential variables.  If fixed
part sizes already present in a progression are removed, one must instead use
the joint convergence in Proposition~\ref{prop:fixed-parts}; the removed
coordinates need not be independent of the remaining coordinates.  In either
case, the corresponding conditional estimate can be obtained by repeating the
finite-marking argument there.

More generally, a signed statistic formed from several residue classes
modulo the same $m$ has an unconditioned characteristic function that is a
finite product of $q$-Pochhammer quotients.  A rigorous extension would need
to state the coefficient-balance condition under which the powers of
$1-q$ cancel and to repeat the uniform Fourier estimates.  In the present
two-progression case this balance is automatic, and the limiting comparison
reduces to two gamma variables whose normalized ratio is beta distributed.

Two natural refinements remain.  First, quantitative versions of the
Fourier estimates should yield a rate of convergence in
\eqref{eq:bias-main}.  Second, allowing the modulus or the number of
marked residue classes to vary with $n$ would require uniform forms of both
the $q$-gamma asymptotic and the deconditioning estimate.  These questions
lie beyond the fixed-parameter regime considered here.

\begin{remark}[On the coprimality normalization]
The condition $\gcd(r,s,m)=1$ is retained in the statement of the main
theorem only because the finite-set theorem of Li and Zhao used in
Proposition~\ref{prop:finite-limit} assumes that the marked part sizes have
greatest common divisor one.  The Boltzmann, $q$-gamma, and conditioning
arguments for the unrestricted partition model do not use this condition.
For the present truncations the full allowed set is $\{1,\ldots,M_N\}$,
whose greatest common divisor is always one, even when $\gcd(r,s,m)>1$.
Thus a general lattice-point version of the finite-set theorem, formulated
under the weaker hypothesis $\gcd(A)=1$ for the full allowed set $A$,
should remove the normalization and leave the same beta constant.  Since
that generalized finite-set lemma is not proved here, we keep the
normalization in the formal theorem rather than silently enlarging the
range of the cited result.
\end{remark}

\section*{Acknowledgment}

The authors thank OpenAI's ChatGPT for assistance with language editing,
exposition, and exploratory discussions concerning the mathematical arguments.
All mathematical statements, proofs, and final revisions were independently
checked and are the sole responsibility of the authors.

\end{document}